\documentclass[12pt]{article}

\usepackage{bookman} 
\usepackage{paralist, mathdots} 

\usepackage{amssymb} 
\usepackage{amsmath,tikz,wasysym,multirow} 

\newtheorem{theorem}{Theorem}[section]
\newtheorem{example}{Example}[section]

\newtheorem{corollary}{Corollary}[section]
\newtheorem{lemma}{Lemma}[section]
\newenvironment{proof}{\textbf{Proof.}}{\qquad $\Box$ \bigskip }

\newcommand{\npmatrix}[1]{\left( \begin{matrix} #1 \end{matrix} \right)}
\newcommand{\R}{\mathbb{R}}

\newcommand{\mc}{\mathcal}

\newcommand{\rank}{\mathrm{rank}}

    \definecolor{helena}{rgb}{.2,.8,.4}
    \definecolor{polona}{rgb}{.8,.2,.2}
   \definecolor{todo}{rgb}{.2,.2,.8}

\begin{document}
\title{Diagonal realizability in the Nonnegative Inverse Eigenvalue Problem}
\author{Thomas J. Laffey, Helena \v Smigoc}
%
%
\bigskip

\maketitle

\begin{abstract}
We show that if a list of nonzero complex numbers $\sigma=(\lambda_1,\lambda_2,\ldots,\lambda_k)$ is the nonzero spectrum of a diagonalizable nonnegative matrix, then $\sigma$ is the nonzero spectrum of a diagonalizable nonnegative matrix of order $k+k^2.$ 
\end{abstract}

\section{Introduction}

The \emph{nonnegative inverse eigenvalue problem} (NIEP) asks which lists of complex numbers can be the spectrum of some entry-wise nonnegative matrix. If a list of complex numbers $\sigma$ is the spectrum of some entry-wise nonnegative matrix $A$, we say that   $\sigma$ is \emph{realizable}, and that $A$ \emph{realises} $\sigma$. The NIEP is a difficult open problem, however, several partial results are known. For other sources of literature on the problem we
refer the reader to the following works and the citations that appear in them: \cite{MR2039754, MR1628398,  MR2232926, MR2742335, MR2098598, MR2112207, MR1473205}.

Motivated by applications in ergodic theory, Boyle and Handelman
\cite{MR1097240} solved a related question: which lists of complex numbers can be the nonzero spectrum of a nonnegative matrix? In particular, they proved that if $\sigma=(\lambda_1,\lambda_2,\ldots,\lambda_n)$ is a list of complex
numbers such that the power sums $s_k=\sum_{i=1}^n \lambda_i^k>0$ for all positive integers $k,$ and $\lambda_1 >|\lambda_i|$ for $i=2,3,\ldots,n$,  then
there exists a nonnegative integer $N$ such that the list obtained by appending
$N$ zeros to $\sigma$ is realizable by a nonnegative $(n+N)\times (n+N)$
matrix. It can be shown that the least $N$ required here is in general not
bounded as a function of $n.$ The proof in \cite{MR1097240}  is not constructive and does not
enable one to determine the size of the $N$ required for realizability in
the general case. A constructive approach to the Boyle and Handelman result that provides a bound on $N$, the number of zeros needed for realizability,  was given by Laffey in \cite{MR2890950}. 

Several other variants of the NIEP have been considered in the literature. One that attracted a lot of attention is the symmetric nonnegative inverse eigenvalue problem (SNIEP), where it is demanded that the realising nonnegative matrix is symmetric. The corresponding question about the nonzero spectrum of a symmetric matrix is open. Unlike in the general case, the number of zeros needed to be added to the nonzero spectrum of a symmetric nonnegative matrix in order to obtain a nonnegative symmetric realisation is bounded by a function of the number of nonzero elements in the the list. 

\begin{theorem}[\cite{MR1350951}]
Let $A \in M_n(\R)$ be a symmetric nonnegative matrix of rank $k$. Then, there exists a symmetric nonnegative matrix $\tilde{A} \in M_{k(k+1)/2}(\R)$ with the same nonzero spectrum as $A$. 
\end{theorem}

This result was used in the first proof that showed the symmetric nonnegative inverse eigenvalue problem is different to the real nonnegative inverse eigenvalue problem, for the problem of determining which lists of real numbers are realizable. The bound provided in the theorem above is believed not to be tight. In fact, examples of lists where one zero added makes the list symmetrically realizable are known, but at present there are no known examples where three or more zeros are required in symmetric realizability. 

In this note, we consider analogous questions for diagonal realizability. In particular, we show that if a list is the nonzero spectrum of a diagonalizable nonnegative matrix with $k$ nonzero eigenvalues, then it can be realised by a nonnegative diagonalizable matrix of order $k(k+1).$ 

The ideas that we use in this note are similar to those in \cite{MR1350951}, where Carath\' eodory's theorem plays a central role. 

\begin{theorem}[Carath\' eodory]\label{thm:Caratheodory} Let $\mc V$ be an $l$-dimensional vector space, and let $v_i \in \mc V,$ $i=1,2,\ldots,p$. Let $\mc K$ be the convex cone generated by $v_1,v_2, \ldots, v_p$. Then each point in $\mc K$ can be expressed as a linear combination, with nonnegative coefficients, of $l$ or fewer of the $v_i$'s. 
\end{theorem}

%

\section{Main Results}

Our approach will depend on the existence of a principal sub-matrix $A_{11}$ of the original matrix $A$ that has the same rank as $A$.
We start by considering the structure of a matrix $A$ with a principal submatrix $A_{11}$ of the same rank as $A$. 

\begin{lemma}\label{lem:full rank subm}
Let  $A=\npmatrix{A_{11} & A_{12} \\A_{21} & A_{22}} \in M_n({\R})$ and $A_{11} \in M_m({\R})$ have the same rank. Then there exists an $m \times (n-m)$ matrix $Q$ such that
$$A=\npmatrix{A_{11} & A_{11}Q \\ A_{21} & A_{21}Q},$$ and $A$ is similar to $$\npmatrix{A_{11}+QA_{21} & 0 \\ A_{21} & 0}.$$
 \end{lemma}

\begin{proof}
Since $\rank(A_{11})=\rank\npmatrix{A_{11}& A_{12}}$, there exists an $m \times (n-m)$ matrix $Q$ so that $A_{12}=A_{11}Q$, and since $\rank(A_{11})$ is equal to the rank of:
\begin{align*}
\npmatrix{A_{11} & A_{11}Q \\A_{21} & A_{22}}\npmatrix{I_{m} & -Q \\ 0 & I_{n-m}}=\npmatrix{A_{11} & 0 \\ A_{21} & A_{22}-A_{21}Q},
\end{align*}
we conclude that $A_{22}=A_{21}Q$. This gives us: 
$$A=\npmatrix{A_{11} & A_{11}Q \\ A_{21} & A_{21}Q}.$$
We compute
$$\npmatrix{I_m & Q \\ 0 & I_{n-m}}\npmatrix{A_{11} & A_{11}Q \\ A_{21} & A_{21}Q}
\npmatrix{I_m & -Q \\ 0 & I_{n-m}}=\npmatrix{A_{11}+QA_{21} & 0 \\ A_{21} & 0}$$
to verify the second part of the statement. 
\end{proof}

Our first application of this lemma given below considers the general case. 

\begin{theorem}\label{thm:general}
Let $$A=\npmatrix{A_{11} & A_{12} \\ A_{21} & A_{22}} \in M_n(\R)$$ be a nonnegative matrix, where $A_{11} \in M_m(\R)$ has rank equal to the rank of $A$ and the rank of $A_{21}$ is equal to $r$. Furthermore, we assume $n > m+mr$.

Then there exists a nonnegative matrix $\tilde{A} \in M_{m+mr}(\R)$ whose nonzero spectrum is the same as the nonzero spectrum of $A$.  Moreover, the Jordan canonical forms of $A$ and $\tilde A$, denoted by $J(A)$ and $J(\tilde A)$ respectively, satisfy: $J(A)=J(\tilde A)\oplus 0_{n-m-mr}$.
\end{theorem}

\begin{proof}
%
By Lemma \ref{lem:full rank subm}
$$A=\npmatrix{A_{11} \\A_{21}}\npmatrix{I_m & Q},$$
hence the nonzero spectrum of $A$ is equal to the nonzero spectrum of 
$$\npmatrix{I_m & Q} \npmatrix{A_{11} \\A_{21}}=A_{11}+QA_{21}.$$
Let us write down the columns of $Q$ and the rows of $A_{21}:$
$$Q=\npmatrix{q_1 & q_2 & \ldots & q_{n-m}} \text{, }A_{21}=\npmatrix{v_1^T \\ v_{2}^T \\ \vdots \\ v_{n-m}^T}.$$ Let $\mc V$ denote the span of $v_i^T$'s, and let $\mc V^{\perp}$ denote its orthogonal complement. The vector space 
 $$\mc S=\{M \in M_m(\R); \, Mv=0 \text{ for all }v \in \mc V^{\perp}\},$$ has dimension $mr$, and contains $q_jv_j^T$, for $i=1,2 \ldots, n-m.$   
By Theorem \ref{thm:Caratheodory} we can write: 
$$QA_{21}=\sum_{j=1}^{mr} \alpha_j q_{r_j}v_{r_j}^T,$$
where $\alpha_j \geq 0.$ We define: 
$$\tilde{Q}=\npmatrix{\alpha_1q_{r_1} & \alpha_2 q_{r_2} & \ldots & \alpha_{mr}q_{r_{mr}}} \text{, }\tilde{A}_{21}=\npmatrix{v_{r_1}^T \\ v_{r_2}^T \\ \vdots \\ v_{r_{mr}}^T},$$
and $$\tilde{A}:=\npmatrix{A_{11} \\ \tilde{A}_{21}}\npmatrix{I_m & \tilde{Q}} \in M_{m+mr}(\R).$$ Note that $\tilde A$ is obtained from $A$ by deleting some rows and corresponding columns of $A$, and then multiplying some of the surviving columns by nonnegative constants. Hence, the nonnegativity of $\tilde A$ is clear from the construction. Furthermore, the nonzero spectrum of $A$ is equal to the nonzero spectrum of $\tilde{A}$, since $$\npmatrix{I_m & Q} \npmatrix{A_{11} \\A_{21}}=\npmatrix{I_m & \tilde{Q}} \npmatrix{A_{11} \\ \tilde{A}_{21}}.$$

We still need to prove the connection between the Jordan forms of $A$ and $\tilde A$. From the construction of $\tilde A$ we see that $\rank(\tilde{A}) \leq \rank(A)$, but since $\tilde A$ contains a principal submatrix $A_{11}$ whose rank is the same as rank of $A$, we conclude that $\rank(A)=\rank(\tilde{A})$. 

Lemma \ref{lem:full rank subm} tells us that $A$ is similar to 
$$A'=\npmatrix{A_{11}+QA_{21} & 0 \\ A_{21} & 0}$$ and $\tilde A$ is similar to 
$$\tilde{A}'=\npmatrix{A_{11}+\tilde{Q}\tilde{A}_{21} & 0 \\ \tilde{A}_{21} & 0}=\npmatrix{A_{11}+QA_{21} & 0 \\ \tilde A_{21} & 0}.$$ We use a permutation similarity on $A'$ to deduce that $A$ is similar to a matrix of the form: 
$$A''=\npmatrix{A_{11}+QA_{21} & 0 &0 \\ \tilde{A}_{21} & 0 & 0 \\ \tilde{\tilde{A}}_{21} & 0 & 0}.$$
Since $\rank(A)=\rank(\tilde{A})$, the rows of $\tilde{\tilde{A}}_{21}$ are linear combinations of rows of $A_{11}+QA_{21}$ and $\tilde{A}_{21}$. Hence we can find an $(n-m-mr) \times m$ matrix $S$  and an $(n-m-mr) \times mr$ matrix $T$ such that a similarity of the form:
$$\npmatrix{I_m & 0 & 0 \\ 0 & I_{mr} & 0 \\ S & T & I_{n-m-mr}}$$
on the matrix $A''$ results in:
$$\npmatrix{A_{11}+QA_{21} & 0 & 0\\ \tilde{A}_{21} & 0 & 0 \\ 0 & 0 & 0 }.$$
Now we know that $A$ is similar to a matrix of the form $\tilde{A}' \oplus 0_{n-m-mr},$ and the relationship between $J(A)$ and $J(A')$ follows. 
\end{proof}

The following lemma allows us to obtain a bound on the size of $A_{11}$ in terms of the rank of $A$ and the number of nonzero eigenvalues of $A$ in the results above. 

\begin{lemma}
Let $A \in M_n(\R)$ have $l$ nonzero eigenvalues and rank $k$. Then $A$ contains a principal submatrix of order $2k-l$ whose rank is equal to the rank of $A$.  
\end{lemma}

\begin{proof}
Let $p(x)=x^n+p_1x^{n-1}+\cdots+p_l x^{n-l}$ be the characteristic polynomial of $A$. Since $A$ has $l$ nonzero eigenvalues, we have $p_l \neq 0$, and $A$ contains a principal $l \times l$ nonzero minor. Using permutation similarity we may assume that the leading $l \times l$ principal minor of $A$, call it $A_{11}$, is not equal to zero. Using another permutation similarity we may assume that the first $k$ rows of $A$ have rank $k$, i.e. are linearly independent:
$$A=\npmatrix{A_{11} & A_{12} & A_{13} \\ A_{21} & A_{22} & A_{23} \\ A_{31} & A_{32} & A_{33} },$$
where $A_{11}$ is invertible, and $\npmatrix{A_{11} & A_{12} & A_{13} \\ A_{21} & A_{22} & A_{23}}$ has full rank $k$. 
This implies that the matrix  $\npmatrix{ A_{12} & A_{13} \\ A_{22} & A_{23}}$ contains $k-l$ linearly independent columns. Using a permutation similarity that leaves the top left $k \times k$ submatrix fixed, we can assure that the $k \times (k+(k-l))$ submatrix of $A $ containing the first $k$ rows and the first $2k-l$ columns contains $k$ linearly independent columns. This implies that the top left $(2k-l) \times (2k-l)$ submatrix of $A$ has rank $k$. 
\end{proof}

The bound given in the above lemma cannot be improved in general as illustrated in the following example.

\begin{example}
Let $D_l$ be an $l \times l$ diagonal matrix with $l$ nonzero diagonal elements. The matrix
 $$A=\npmatrix{D_l & 0 & 0 \\ 0 & 0_{k-l} & I_{k-l} \\ 0 & 0 & 0_{k-l}}$$ 
has order $2k-l$, rank $k$ and $l$ nonzero eigenvalues. It is easy to check that $A$ has no proper principal submatrices with rank $k$. 
\end{example}

\begin{corollary}
Let $A \in M_n(\R)$ be a nonnegative matrix with $l$ nonzero eigenvalues and rank $k$. Then there exists a nonnegative matrix $\tilde{A}$ of order  $\tilde n={(2k-l)+(2k-l)^2}$, whose nonzero spectrum is the same as the nonzero spectrum of $A$ and whose Jordan canonical form $J(\tilde{A})$ satisfies: $J(A)=J(\tilde{A}) \oplus 0_{n-\tilde n}$.  
\end{corollary}

\begin{corollary}
Let $A \in M_n(\R)$ be a diagonalizable nonnegative matrix of rank $k$ and $n\geq k+k^2$. Then there exists a diagonalizable nonnegative matrix $\tilde{A} \in M_{k+k^2}(\R)$, whose nonzero spectrum is the same as the nonzero spectrum of $A$.  
\end{corollary}

\begin{proof}
For a diagonalizable matrix $A$ the rank of $A$ is equal to the number of nonzero eigenvalues of $A$. 
\end{proof}

\section{The Nonnegative Inverse Elementary Divisor Problem}

The nonnegative inverse elementary divisor problem (NIEDP) asks for a given realizable spectrum $\sigma=(\lambda_1,\lambda_2,\ldots,\lambda_n)$, what are the possible Jordan forms of realising matrices. Of course, if $\sigma$ has no repeated entries, this problem reduces to the NIEP for $\sigma$. It is conjectured, that if $\sigma$ is realizable, then it is realizable by a nonnegative nonderogatory matrix, but this appears to be still open.

 Minc \cite{MR630033} proved that if $\sigma$ is diagonalizably realizable by a positive matrix $A$, then for every Jordan form $J$ with spectrum $\sigma$, $\sigma$ is realizable by a positive matrix similar to $J$. This result is conjectured to hold, if we relax the condition that $A$ is positive to the assumption that $A$ is nonnegative, but this also seems to be open at present.

%

Now consider the classic example $$\sigma(t)=(3+t,3-t, -2,-2,-2).$$ We ask the questions what is the minimal $t$ for which $\sigma(t)$ is realizable by a nonnegative matrix with a given Jordan canonical form $J_i(t)$ associated with $\sigma(t)$, where $J_1(t)$ is a diagonal matrix, $J_2(t)$ is the Jordan canonical form with the minimal polynomial of degree $4$, and $J_3(t)$
 is nonderogatory. We will denote the minimal $t$ in each case by $t_i$.

It is shown in \cite{2017arXiv170108651C} that $\sigma(t)$ is realizable by a diagonalizable nonnegative matrix only for $t \geq 1,$ i.e. $t_1=1$.  In this case, the condition for diagonalizable realizability and symmetric realizability coincide. Also, in the same paper it is shown that $$\sigma=(3+t,3-t,-1.9,-2,-2.1)$$ is realizable for $t \geq \frac{1}{10}\sqrt{120\sqrt{3166}-3899} \approx 0.435$, while by McDonald-Neumann inequality, given in \cite{MR1780536}, $t \geq 0.9$ is necessary for symmetric realizability.  
 
 On the other hand, $\sigma(t)$ is realizable for $t \geq \sqrt{16 \sqrt{6}-39}\approx 0.438$. This is shown in \cite{MR1733536}, where a nonderogatory matrix with spectrum $\sigma(t_3),$ $t_3=\sqrt{16 \sqrt{6}-39}$, is provided. 

Here we present the matrix
$$A(t)=\left(
\begin{array}{ccccc}
 0 & 2 & \frac{1}{2} & 0 & 0 \\
 2 & 0 & \frac{1}{2} & 0 & 0 \\
 \frac{256}{t^2+7}-32 & \frac{256}{t^2+7}-32 & 0 & 1 & 0 \\
 0 & 0 & \frac{t^4+78 t^2-15}{2 \left(t^2+7\right)} & 0 & \frac{1}{2} \sqrt{2 t^2+30} \\
 0 & 0 & \frac{2 \sqrt{2} \left(3 t^4+58 t^2+3\right)}{\left(t^2+7\right) \sqrt{t^2+15}} & \frac{1}{2} \sqrt{2 t^2+30} & 0 \\
\end{array}
\right).$$
$A(t)$ has eigenvalues $(3+t,3-t,-2,-2,-2)$, it is nonnegative, and has Jordan canonical form $J_2(t)$ for $\sqrt{16 \sqrt{6}-39}\leq t$.
This shows that $t_2=t_3$.

\bibliographystyle{plain}
\bibliography{../Bib/NIEP}
\end{document}